\documentclass[review]{elsarticle}
\usepackage{lineno}
\usepackage[latin1]{inputenc}
\usepackage{times}
\usepackage{amssymb,mathrsfs, amsmath}
\usepackage{amsmath,amsthm}
\usepackage{amssymb}
\usepackage{latexsym}
\usepackage{amsfonts}
\usepackage{mathrsfs}
\usepackage{epsfig}
\usepackage{hyperref}
\usepackage{graphicx}
\usepackage{graphics}

\newtheorem{thm}{Theorem}[section]

\newtheorem{defn}{Definition}[section]
\newtheorem{prop}{Proposition}[section]

\newtheorem{lem}{Lemma}[section]

\newtheorem{rem}{Remark}[section]

\newtheorem{cor}{Corollary}[section]

\newtheorem{exmpl}{Example}[section]

\journal{Journal of Algebra  }

\begin{document}
\begin{frontmatter}

\title{Equivariant one-parameter deformations of   Lie triple systems}

%%% Group authors per affiliation:
%%\author{\fnref{myfootnote}}
%%\address{}
%%\fntext[myfootnote]{This author is supported by CSIR,\textsc{India}.}
%\author{RB Yadav \fnref{myfootnote}\corref{mycorrespondingauthor}}
%\address{Indian Statistical Institute Tezpur, Assam 784028, \textsc{India}}
%\ead{rbyadav15@gmail.com}
%\author{Goutam Mukherjee\fnref{myfootnote}\corref{mycorrespondingauthor}}
%\address{Stat-Math Division,  Indian Statistical Institute, Kolkata, 700108, \textsc{India}}
%\ead{goutam@isical.ac.in}
%\fntext[myfootnote]{This author is supported by CSIR,\textsc{India}.}
%
%%
\author{RB Yadav\fnref{}\corref{mycorrespondingauthor}}
%\address{Sikkim University, Gangtok, Sikkim, 737102, \textsc{India}}
\cortext[mycorrespondingauthor]{Corresponding author}
\ead{rbyadav15@gmail.com}
\author{Namita Behera\corref{mycoauthor}}
%\address{Sikkim University, Gangtok, Sikkim, 737102, \textsc{India}}
\ead{nbehera@cus.ac.in}
\author{ Rinkila Bhutia\corref{mycoauthor}}
\address{Sikkim University, Gangtok, Sikkim, 737102, \textsc{India}}
\ead{rbhutia@cus.ac.in}

\begin{abstract}
In this article, we introduce equivariant formal deformation theory of  Lie triple systems. We  introduce  an equivariant deformation cohomology of Lie triple systems  and using this we study the  equivariant formal deformation theory of Lie triple systems.

\end{abstract}

\begin{keyword}
\texttt{Lie triple system, Group actions, Yamaguti cohomology, equivariant formal deformations, equivariant cohomology}
\MSC[2020] 17B99 \sep 16S80 \sep 13D10 \sep 13D03  \sep 14D15 \sep 16E40 \sep 55N91
\end{keyword}

\end{frontmatter}

%\linenumbers

%\maketitle

\section{Introduction}\label{rbsec1}
 Lie triple systems were first noted by E. Cartan in his studies on totally geodesic submanifolds of Lie groups and on symmetric spaces.
Lie triple systems were studied from the algebraic point of view by Jacobson \cite{NJ,NJ1} and Lister \cite{WGL}. Simpler axioms were given by Yamaguti \cite{KY4}, who has also studied these and more general systems \cite{KY2,KY3}.

The deformation is a tool to study a mathematical object by deforming it into a family of the same kind of objects depending on a certain parameter. Algebraic deformation theory was introduced by Gerstenhaber for rings and algebras in a series of papers \cite{MG1},\cite{MG2},\cite{MG3}, \cite{MG4}, \cite{MG5}. Kubo and Taniguchi introduced deformation theory for Lie triple systems \cite{Kub-Tani}. They studied one-parameter formal deformations and established the connection between the cohomology groups and deformations: the suitable cohomology groups for the deformation theory of associative algebras and Lie triple systems are the Hochschild cohomology \cite{Hochs} and the Yamaguti cohomology \cite{KY1}, respectively. Equivariant deformation theory of associative algebras has been studied in \cite{gmrb}.

Purpose of this paper is to  introduce  equivariant deformation cohomology  and equivariant formal deformation theory of Lie triple systems.  Organization of the paper is as follows. In Section \ref{rbsec2}, we recall some definitions and results. In Section \ref{rbsec3}, we introduce equivariant deformation complex and equivariant deformation cohomology of a Lie triple system. In Section \ref{rbsec4}, we introduce equivariant deformation of a Lie triple system. In this section  we prove  that obstructions to equivariant deformations are equivariant cocycles. Also, in this section we give an example of an equivariant formal deformation of a Lie triple system. In Section \ref{rbsec5}, we study equivalence of two equivariant deformations and rigidity of an equivariant Lie triple system.

\section{Preliminaries}\label{rbsec2}
In this section, we recall definitions of Lie triple system, module over a Lie triple system,  Yamaguti cohomology  and formal one parameter deformation of Lie triple system.  Throughout the paper we denote a fixed field  by k and a finite group by G. For  k-modules $M_1,M_2,\cdots, M_n$, we denote  $x_1\otimes \cdots\otimes x_n\in M_1\otimes\cdots \otimes M_n$  by $(x_1,\cdots, x_n)$.  Also, we denote the ring of formal power series with coefficients in k by $k[[t]]$.
\begin{defn}
   A Lie triple system(Lts) is a vector space T over k with a k-linear map $\mu:T\otimes T\otimes T\to T$  satisfying   (if we write $\mu(a\otimes b\otimes c)=[abc]$ )

  \begin{equation}\label{LT1}
    [aab]=0,
  \end{equation}
 \begin{equation}\label{LT2}
    [abc]+[bca]+[cab]=0,
  \end{equation}
   \begin{equation}\label{LT3}
    [ab[cde]]=[[abc]de]+[c[abd]e]+[cd[abe]],
  \end{equation}
  for $a,b,c,d,e\in T.$
$[\;]$ is called the ternary operation of the Lie triple system T.
\end{defn}
\begin{defn}
We say that a vector space V is a  module over a Lts T  provided that $E_V := T \oplus V$ possesses the structure of a Lie triple system such that: (a) T is a Lie triple subsystem of $E_V$, (b) for $a, b, c \in E_V$, $[a, b, c] \in V$ if any one of a, b, c lies in V, and (c) [a, b, c] = 0 if any two of a, b, c lie in V. We also say that V is a T-module.\\
An  equivalent definition of T-module is given as follows:\\
    A module V over T is a k-vector space  with three actions $\mu_1$,  $\mu_2$ and $\mu_3$  (left, right and lr)  of $T\otimes T$ on V, that is, three linear maps $T\otimes T\otimes V\to V$, (for simplicity we denote  the three  actions by the same symbol [], one can differentiate  them from   the context, for $a,b\in T$, $c\in V$ we  write $\mu_1(a\otimes b\otimes c)=[abc]$, $\mu_2(a\otimes b\otimes c)=[cab]$, $\mu_3(a\otimes b\otimes c)=[acb]$) such that

    \begin{equation}\label{LTM1}
    [abc]=-[bac],
  \end{equation}
 \begin{equation}\label{LTM2}
    [abc]+[bca]+[cab]=0,
  \end{equation}
   \begin{equation}\label{LTM3}
    [ab[cde]]=[[abc]de]+[c[abd]e]+[cd[abe]],
  \end{equation}
  where atmost one of a,b,c,d,e is in V and remaining are in T.
\end{defn}
Clearly, a Lts T itself can be considered as a module over itself with respect to the actions $\mu_1$,  $\mu_2$ and $\mu_3$ given by  $\mu_1(a\otimes b\otimes c)=[abc]$, $\mu_2(a\otimes b\otimes c)=[cab]$, $\mu_3(a\otimes b\otimes c)=[acb]$). Here $[\;]$ denotes the ternary operation of T.
 \begin{rem} If V is a module over T, then there exists a linear map $\theta:T\otimes T\to End(V)$ satisfying the following conditions:
       \begin{equation}\label{LTR1}
      \theta(c,d)\theta(a,b)-\theta(b,d)\theta(a,c)-\theta(a,[bcd])+D(b, c)\theta(a, d)=0
    \end{equation}
    \begin{equation}\label{LTR2}
      \theta(c,d)D(a,b)-D(a, b)\theta(c, d)+\theta([abc],d)+\theta(c,[abd])=0
    \end{equation}
    where $D(a, b)=\theta(b,a)-\theta(a,b)$.  One can readily verify this by defining $\theta (a\otimes b)(v)=\mu_2(a\otimes b\otimes v)=[vab]$ , for all $a,b\in T$, $v\in V.$
 \end{rem}

\begin{exmpl}\label{lte2}
Space of all $n\times n$ matrices   is a Lie triple system with the ternary operation $[\;]$ defined by  $[ABC]=[[AB]_1C]_1$, for any $n\times n$ matrices A,B, C,  where $[UV]_1=UV-VU$, for any $n\times n$ matrices U, V.  We denote it by $M(n)$.\\
Space of all $n\times n$ skew symmetric matrices is a Lie triple system with the ternary operation $[\;]$ defined by  $[ABC]=[[AB]_1C]_1$, for any $n\times n$ skew symmetric matrices A, B, C,  where $[UV]_1=UV-VU$,  for any $n\times n$ matrices U,V.  We denote it by $M_{sk}(n)$.\\
 In general, every Lie algebra L with Lie product $[\;]_2$ is  a Lie triple system with ternary operation  $[\;]$ defined by $[abc]=[[ab]_2c]_2$, for any $a, b, c\in L.$
\end{exmpl}
\begin{exmpl}\label{lte3}
   Space of all $n\times n$ symmetric matrices is a Lie triple system with the ternary operation $[\;]$ defined by  $[ABC]=[[AB]_1C]_1$, for any $n\times n$ symmetric matrices A, B, C, where $[UV]_1=UV-VU$,  for any $n\times n$ matrices U,V.  We denote it by $M_{s}(n)$.\\
\end{exmpl}
\begin{exmpl}\label{rbe1}
\cite{NJ}  Let $T_n$ be the vector space of  $(n+1)\times (n+1)$  matrices with entries in a field $\mathbb{F}$ and spanned by  the set $\{g_i: g_i=e_{i,n+1}-e_{n+1,i}, i=1,2,\cdots, n\}$.  Define a linear map $[\;]:T_n\otimes T_n\otimes T_n\to T_n$ by $[g_ig_jg_l]=\delta_{li}g_j-\delta_{lj}g_i$. Here we denote $[\;](a\otimes b\otimes c)$ by $[abc]$, for all $a,b,c\in T_n$.  It can be readily verified that

  \begin{equation}\label{ltseqn1}
 [g_{i_1}g_{i_2}g_{i_3}] = -[g_{i_2}g_{i_1}g_{i_3}]
\end{equation}
 \begin{equation}\label{ltseqn2}
  [g_{i_1}g_{i_2}g_{i_3}]+[g_{i_2}g_{i_3}g_{i_1}]+[g_{i_3}g_{i_1}g_{i_2}] =0
\end{equation}
  \begin{equation}\label{ltseqn3}
 [g_{i_1}g_{i_2}[g_{i_3}g_{i_4}g_{i_5}] = [[g_{i_1}g_{i_2}g_{i_3}]g_{i_4}g_{i_5}]+[g_{i_3}[g_{i_1}g_{i_2}g_{i_4}]g_{i_5}]+[g_{i_3}g_{i_4}[g_{i_1}g_{i_2}g_{i_5}]],
\end{equation}
 for $1\le i_j\le n,$ with $1\le j\le 5$.  From  \ref{ltseqn1}-\ref{ltseqn3} and linearity of $[\;]$, we conclude that $T_n$ with the ternary operation $[\;]$ is a Lie triple system. $T_n$ is called the Meson triple system .
\end{exmpl}
\begin{exmpl}\label{lte4}
Let M(p,q) be the vector space of  all $p\times q$ matrices with entries in a field $\mathbb{F}$. M(p,q) is a Lie triple systems with a ternary operation $[\;]$ defined by $[ABC]=(AB^t-BA^t)C+C(B^tA-A^tB)$, for any $A, B, C\in M(p,q)$. Here $A^t$ denotes transpose of A, for any $A\in M(p,q)$.
\end{exmpl}
\begin{exmpl}\label{lte5}
  Let T be a Lie triple system and S be any nonempty set. Then the set $T^S$ of all functions $f$ from S to T is a Lie triple system with respect to a ternary operation given by  $[f_1f_2f_3](x)=[f_1(x)f_2(x)f_3(x)].$
\end{exmpl}
\begin{defn}[Yamaguti Cohomlogy]
  Let T be a Lie triple system and V be a module over T. From \cite{KY1}, we recall that for each $n\ge 0$  a k-vector space  $C^{2n+1}(T;V)$ is defined as follows:  For $n\ge 1$, $C^{2n+1}(T;V)$ consists of  those  $ f\in Hom_k(T^{\otimes (2n+1)},V),$  which  satisfy
   $$f(x_1,\cdots, x_{2n-2},x,x,y)=0$$
   and $$f(x_1,\cdots, x_{2n-2},x,y,z)+f(x_1,\cdots, x_{2n-2},y,z,x)+f(x_1,\cdots, x_{2n-2},z,x,y)=0,$$
   and $C^1(T;V)=Hom_k(T,V)$. A  k-linear map $\delta^{2n-1}:C^{2n-1}(T;V)\to C^{2n+1}(T;V)$ is defined by
  \begin{eqnarray*}
    &&\delta ^{2n-1}f(x_1,\cdots, x_{2n+1})\\
     &=& \theta(x_{2n},x_{2n+1})f(x_1,\cdots,x_{2n-1})-\theta(x_{2n-1},x_{2n+1})f(x_1,\cdots,x_{2n-2},x_{2n})\\
    &&+\sum_{k=1}^{n}(-1)^{k+n}D(x_{2k-1},x_{2k})f(x_1,\cdots,\widehat{ x_{2k-1}}\widehat{x_{2k}}, \cdots, x_{2n+1}) \\
     && +\sum_{k=1}^{n}\sum_{j=2k+1}^{2n+1}(-1)^{n+k+1}f(x_1,\cdots,\widehat{ x_{2k-1}}\widehat{x_{2k}}, \cdots,[x_{2k-1}x_{2k}x_j],\cdots, x_{2n+1}).
  \end{eqnarray*} This gives a cochain complex  $(C^{\ast}(T;V),\delta)$, cohomology of which is denoted by $H^{\ast}(T;V)$ and called as Yamaguti cohomology of T with coefficients in V. Since T is a module over itself. So we can consider Yamaguti cohomology $H^{\ast}(T;T)$.\\

\end{defn}

\section{Group actions and equivariant Yamaguti Cohomology}\label{rbsec3}
Let T be a Lie triple system  with its ternary operation $\mu (a\otimes b\otimes c) = [abc]$ and G be a finite group. The group G is said to act on T from the left if there exists a function
$$\phi:G\times  T \to T\;\;\;\; (g, a)\mapsto \phi(g, a) = ga$$
satisfying the following conditions.
\begin{enumerate}
  \item $ex = x$ for all $x \in T$, where $e \in G$ is the group identity.
  \item $g_1(g_2x) = (g_1g_2)x$ for all $g_1, g_2 \in G$ and $x \in T$.
  \item  For every $g \in G$, the left translation $\phi_g = \phi (g, ) : T \to T,$ $a \mapsto ga$ is a linear
map.
  \item For all $g \in G$ and $a, b, c\in T$,  $\mu(ga,gb,gc) = g\mu(a, b,c) = g[abc]$, that is, $\mu$ is equivariant with respect to the diagonal action on $T \otimes T\otimes T$.
\end{enumerate}
We denote an action as above    by $(G,T)$. We call Lie triple system (Lts) with an action of a group G as G-Lts.
\begin{prop}
  Let G be a finite group and T be a Lie triple system. Then G acts on T if and only if there exists a group homomorphism
  $$ \psi: G \to Iso_{Lts}(T,T),\;\; g\mapsto  \psi(g) = \phi_g$$
from the group G to the group of Lie triple system isomorphisms from T to T.
\end{prop}
\begin{proof}
  For an action $(G,A)$, we define a map $\psi:G\to Iso_{Lts}(T,T)$ by $\psi(g)=\phi_g.$ One can verify easily that $\psi$ is a  group homomorphism.
  Now, let $\psi:G\to Iso_{Lts}(T,T)$ be a group homomorphism. Define a map $G\times T\to T$ by $(g,a)\mapsto \psi(g)(a).$ It can be easily seen that this is an action of G on the Lts  T.
\end{proof}
\begin{exmpl}
  Consider the Lie triple system of $n\times n$ skew symmetric matrices  $M_{sk}(n)$ as in Example \ref{lte2}. Consider the group $\mathbb{Z}_2=\{\bar{0},\bar{1}\}$.  $\mathbb{Z}_2$ acts on $M_{sk}(n)$ by $\bar{0} A=A$, $\bar{1} A=-A$, for any $A\in M_{sk}(n).$
\end{exmpl}
\begin{exmpl}\label{elts1}
 From Example \ref{rbe1},  $T_2$ is a Lie triple system with  the ternary operation $\mu :T_2\otimes T_2 \otimes T_2\to T_2$,  $(a,b,c)\mapsto \mu (a,b,c)=[a,b,c]$, defined by $\mu(g_i,g_j,g_l)=\delta_{li}g_j-\delta_{lj}g_i$. We have $T_2=span\{g_1,g_2\}$. Define   $\phi :\mathbb{Z}_2\times T_2 \to T_2$,    $(g, a)\mapsto \phi(g, a) = ga$,  by  $\bar{0}g_1=g_1$,  $\bar{0}g_2=g_2$, $\bar{1}g_1=g_2$, $\bar{1}g_2=g_1$. We have
  \begin{eqnarray}\label{ltse5}
    \bar{1}\mu(g_i,g_j,g_l)&=&\bar{1}\begin{cases}
                       g_j, & \mbox{if } i=l,\; j\ne l \\
                       -g_i, & \mbox{if } j=l,\; i\ne l  \\
                       0, & \mbox{otherwise}.
                        \end{cases}\\
 &=&\begin{cases}
                       g_l, & \mbox{if } i=l,\; j\ne l \\
                       -g_l, & \mbox{if } j=l,\; i\ne l  \\
                       0, & \mbox{otherwise}.
                     \end{cases}
                     \end{eqnarray}
  \begin{eqnarray}\label{ltse6}
     \mu(\bar{1}g_i,\bar{1}g_j,\bar{1}g_l) &=& \begin{cases}
                      \mu(g_j,g_l, g_j), & \mbox{if } i=l,\; j\ne l \\
                       \mu(g_l,g_i,g_i), & \mbox{if } j=l,\; i\ne l  \\
                       0, & \mbox{otherwise}.
                     \end{cases} \\
     &=& \begin{cases}
                       g_l, & \mbox{if } i=l,\; j\ne l \\
                       -g_l, & \mbox{if } j=l,\; i\ne l  \\
                       0, & \mbox{otherwise}.
                     \end{cases}
  \end{eqnarray}

From \ref{ltse5}, \ref{ltse6} and trilinearity of $\mu$ one can conclude that $\mathbb{Z}_2$ acts on the Lts $T_2$.
\end{exmpl}

\begin{exmpl}
 Consider the Lie triple system of $p\times p$  matrices $M(p,p)$ as in Example \ref{lte4}. One can readily verify that $\mathbb{Z}_2$ acts on $M(p,p)$ by $\bar{0}A=A$, $\bar{1}A=A^t$, for any $A\in M(p,q)$.
\end{exmpl}
\begin{defn}
  Let T be G-Lts. A G-module over T is a module V over T such that G acts on V, and  the three (left, right and lr)  actions  of $T\otimes T$ on  V $\mu_1$, $\mu_2$, $\mu_3$ are G-equivariant.
\end{defn}
Clearly, every G-Lts T is a G-module over itself.\\

Define, $\forall n\ge 0$, $$C_G^{2n+1}(T;V)=\{c\in C^{2n+1}(T;V): c(gx_1,\cdots, gx_{2n+1})=gc(x_1,\cdots, x_{2n+1}), \text{ for all}\; g\in G\}$$
An element in $C_G^{2n+1}(T;V)$ is called an invariant (2n+1)-cochain.  Clearly, $C_G^{2n+1}(T;V)$ is a vector subspace of $C^{2n+1}(T;V)$.

 We have following lemma.
\begin{lem}\label{rbnb1}
 c is  an invariant (2n-1)-cochain implies that $\delta^{2n-1}(c)$ is an invariant (2n+1)-cochain.
\end{lem}
\begin{proof}
  Let $c \in C_G^{2n-1}(T;V)$ and $g\in G.$ By definition, we have $$c(gx_1,\cdots,gx_{2n-1})=gc(x_1,\cdots,x_{2n-1}),$$ $\forall (x_1,\cdots,x_{2n-1})\in T^{\otimes(2n-1)}.$ Also,
  \begin{eqnarray}
     &&\delta^{2n-1}(c)(gx_1,\cdots,gx_{2n+1})\notag \\
    &=& \theta(gx_{2n},gx_{2n+1})c(gx_1,\cdots,gx_{2n-1})-\theta(gx_{2n-1},gx_{2n+1})c(gx_1,\cdots,gx_{2n-2},gx_{2n})\notag\\
    &&+\sum_{k=1}^{n}(-1)^{k+n}D(gx_{2k-1},gx_{2k})c(gx_1,\cdots,\widehat{g x_{2k-1}}\widehat{gx_{2k}}, \cdots, gx_{2n+1})\notag \\
     && +\sum_{k=1}^{n}\sum_{j=2k+1}^{2n+1}(-1)^{n+k+1}c(gx_1,\cdots,\widehat{ gx_{2k-1}}\widehat{gx_{2k}}, \cdots,[gx_{2k-1}gx_{2k}gx_j],\cdots, gx_{2n+1})\notag\\
     &=&g\delta^{2n-1}(c)(x_1,\cdots,x_{2n+1})
  \end{eqnarray}
 Hence, $c \in C_G^{2n-1}(T;V)$ implies that $\delta^{2n-1}c \in C_G^{2n+1}(T;V)$.
\end{proof}

From \ref{rbnb1}, we conclude that  $(C_G^{\ast}(T;V), \delta)$  is  a cochain complex.
\begin{defn}
  We call the cochain complex $(C_G^{\ast}(T;V), \delta)$ as equivariant Yamaguti cochain complex of G-Lts T with coefficients in G-module V. We call the cohomology  $H^*_G(T,V)=H^*(C_G^{\ast}(T;V))$ of this complex as  equivariant Yamaguti cohomology of T. For $V=T$, we denote the cohomology $H^*_G(T,T)$ by $H^*_G(T)$.
\end{defn}

\section{ Equivariant deformation of  a Lie triple system}\label{rbsec4}
\begin{defn}\label{rb2}

Let T be a Lie triple system with an action of G. We denote the space of all formal power series with coefficients in T by $T[[t]]$.  An equivariant formal one-parameter deformation of a G-Lts T is a $k[[t]]$-linear map  $$\mu_t : T[[t]]\otimes T[[t]]\otimes T[[t]]\to T[[t]]$$ satisfying the following properties:
\begin{itemize}
  \item[(a)]  $\mu_t(a,b,c)=\sum_{i=0}^{\infty}\mu_i(a,b,c) t^i$, for all $a,b,c\in T$, where $\mu_i:T\otimes T\otimes T\to T$ are k-linear and $\mu_0(a,b,c)=\mu(a,b,c)=[abc]$ is the original ternary operation on T.
  \item [(b)] For every $g\in G$, $$\mu_i(ga,gb,gc)=g\mu_i(a,b,c), \;\;\forall a,b,c\in T,$$ for every $i\ge 0.$ This is equivalent to saying that $\mu_i\in Hom_k^G(T\otimes T\otimes T,T),$ for all $i\ge 0$
      \item[(c)]

  \begin{equation}\label{DLT1}
    \mu_t(a,a,b)=0,
  \end{equation}
 \begin{equation}\label{DLT2}
    \mu_t(a,b,c)+\mu_t(b,c,a)+\mu_t(c,a,b)=0,
  \end{equation}
   \begin{equation}\label{DLT3}
    \mu_t(a,b,\mu_t(c,d,e))=\mu_t(\mu_t(a,b,c),d,e)+\mu_t(c,\mu_t(a,b,d),e)+\mu_t(c,d,\mu_t(a,b,e)),
  \end{equation}
  for all  $a,b,c,d,e\in T$
\end{itemize}

The equations  \ref{DLT1},\ref{DLT2} and \ref{DLT3}  are equivalent to following equations, respectively.
\begin{equation}\label{rbeqn1}
  \mu_r(a,a,b)=0, \;\text{for all}\; a,b\in T, \;r\ge 0.
  \end{equation}
  \begin{equation}\label{rbeqn2}
    \mu_r(a,b,c)+\mu_r(b,c,a)+\mu_r(c,a,b)=0, \;\text{for all}\; a,b,c\in T, \;r\ge 0.
   \end{equation}
   \begin{eqnarray}\label{rbeqn3}
      &&\sum_{i+j=r}\mu_i(a,b,\mu_j(c,d,e))\notag\\
      &=& \sum_{i+j=r}\{\mu_i(\mu_j(a,b,c),d,e)+\mu_i(c,\mu_j(a,b,d),e) \notag\\
       && +\mu_i(c,d,\mu_j(a,b,e))\}; \;\text{for all}\; a,b,c,d,e\in T,\;  r\ge 0
   \end{eqnarray}

\end{defn}
 Now we define equivariant formal deformations  of finite order.

 \begin{defn}\label{rb3}
Let T be a Lie triple system with an action of G.  An equivariant formal one-parameter deformation of order n  of a G-Lts T is a $k[[t]]$-linear map  $$\mu_t : T[[t]]\otimes T[[t]]\otimes T[[t]]\to T[[t]]$$ satisfying the following properties:
\begin{itemize}
  \item[(a)]  $\mu_t(a,b,c)=\sum_{i=0}^{n}\mu_i(a,b,c) t^i$, for all $a,b,c\in T$, where $\mu_i:T\otimes T\otimes T\to T$ are k-linear and $\mu_0(a,b,c)=\mu(a,b,c)=[abc]$ is the original ternary operation on T.
  \item [(b)] For every $g\in G$, $$\mu_i(ga,gb,gc)=g\mu_i(a,b,c), \;\;\forall a,b,c\in T,$$ for every $i\ge 0.$ This is equivalent to saying that $\mu_i\in Hom_k^G(T\otimes T\otimes T,T),$ for all $i\ge 0$
      \item[(c)]

  \begin{equation}\label{FDLT1}
    \mu_t(a,a,b)=0,
  \end{equation}
 \begin{equation}\label{FDLT2}
    \mu_t(a,b,c)+\mu_t(b,c,a)+\mu_t(c,a,b)=0,
  \end{equation}
   \begin{equation}\label{FDLT3}
    \mu_t(a,b,\mu_t(c,d,e))=\mu_t(\mu_t(a,b,c),d,e)+\mu_t(c,\mu_t(a,b,d),e)+\mu_t(c,d,\mu_t(a,b,e)),
  \end{equation}
  for all  $a,b,c,d,e\in T$
\end{itemize}
\end{defn}
\begin{rem}\label{rbrem1}
  \begin{itemize}
    \item For $r=0$, conditions \ref{rbeqn1}-\ref{rbeqn3} are equivalent to the fact that T is a Lie triple system.
    \item For $r=1$, conditions \ref{rbeqn1}, \ref{rbeqn2} and \ref{rbeqn3} are equivalent to $\;\;\mu_1(a,a,b)=0,\;$ $ \mu_1(a,b,c)+\mu_1(b,c,a)+\mu_1(c,a,b)=0$ and \begin{eqnarray*}\label{rrbeqn1}
      0&=&-\mu_1(a,b,[cde])-[ab\mu_1(c,d,e)]\\
      &&+ \mu_1([abc],d,e)+[\mu_1(a,b,c)de] +\mu_1(c,[abd],e)+[c\mu_1(a,b,d)e] \\
       && +\mu_1(c,d,[abe])+[cd\mu_1(a,b,e)]\\
       &=&\delta^3\mu_1(a,b,c,d,e); \;\text{for all}\; a,b,c,d,e\in T
   \end{eqnarray*}
         %From (ii), (iv) and (vi), we have $\mu_1$, $\nu_1$ and $\phi$ are G-equivariant.
          Thus for $r=1$,  \ref{rbeqn1}-\ref{rbeqn3} are equivalent to saying that $\mu_1\in C^3_G(T,T)$  and is a 3-cocycle. In general, for $r\ge 0$, $\mu_r$ is just a 3-cochain in $C^3_G(T,T).$
  \end{itemize}
\end{rem}
\begin{exmpl}
  Consider $\mathbb{Z}_2$-Lts $T_2$  as in Example \ref{elts1}. Define a k-linear map $\mu_2:T_2\otimes T_2\otimes T_2\to T_2$ by $\mu_2(g_i,g_j,g_p)=\delta_{jp}g_j-\delta_{ip}g_i$. Define $\mu_t=\mu+\mu_2t^2$.  One can readily verify that

  \begin{equation}\label{dltseqn1}
 \mu_2(g_i,g_j,g_p) = -\mu_2(g_j,g_i,g_p)
\end{equation}
 \begin{equation}\label{dltseqn2}
  \mu_2(g_i,g_j,g_p)+\mu_2(g_j,g_p,g_i)+\mu_2(g_p,g_i,g_j)=0
\end{equation}
   \begin{eqnarray}\label{dltse5}
    \bar{1}\mu_2(g_i,g_j,g_p)&=&\bar{1}\begin{cases}
                       g_i, & \mbox{if } i=p,\; j\ne p\\
                       g_j, & \mbox{if } j=p,\; i\ne p \\
                       0, & \mbox{otherwise}.
                        \end{cases}\\
 &=&\begin{cases}
                       -g_j, & \mbox{if } i=p,\; j\ne p\\
                       g_i, & \mbox{if } j=p,\; i\ne p \\
                       0, & \mbox{otherwise}.
                     \end{cases}
                     \end{eqnarray}
  \begin{eqnarray}\label{dltse6}
     \mu_2(\bar{1}g_i,\bar{1}g_j,\bar{1}g_p) &=& \begin{cases}
                      \mu_2(g_j,g_p, g_j), & \mbox{if } i=p,\; j\ne p\\
                       \mu_2(g_p,g_i,g_i), & \mbox{if } j=p,\; i\ne p \\
                       0, & \mbox{otherwise}.
                     \end{cases} \\
     &=& \begin{cases}
                       -g_j, & \mbox{if } i=p,\; j\ne p\\
                       g_i, & \mbox{if } j=p,\; i\ne p \\
                       0, & \mbox{otherwise}.
                     \end{cases}
  \end{eqnarray}
  \begin{eqnarray}\label{dltseqn3}
  &&\mu_2(g_i,g_j,\mu_2(g_p,g_l,g_m))\notag\\
  & =& \mu_2(\mu_2(g_i,g_j,g_p),g_l,g_m)+\mu_2(g_p,\mu_2(g_i,g_j,g_l),g_m)+\mu_2(g_p,g_l,\mu_2(g_i,g_j,g_m))\notag\\
\end{eqnarray}
\begin{eqnarray}\label{dltseqn3}
 &&\mu_2(g_i,g_j,\mu(g_p,g_l,g_m)) +\mu(g_i,g_j,\mu_2(g_p,g_l,g_m))\notag\\
 &=& \mu(\mu_2(g_i,g_j,g_p),g_l,g_m)+ \mu_2(\mu(g_i,g_j,g_p),g_l,g_m)+\mu(g_p,\mu_2(g_i,g_j,g_l),g_m)\notag\\
 &&+\mu_2(g_p,\mu(g_i,g_j,g_l),g_m)+\mu(g_p,g_l,\mu_2(g_i,g_j,g_m))+\mu_2(g_p,g_l,\mu(g_i,g_j,g_m))\notag\\
\end{eqnarray}
From Equations \ref{dltseqn1}-\ref{dltseqn3}, one can conclude that $\mu_t$ is an equivariant deformation of $T_2$ of order 2.
\end{exmpl}
\begin{defn}
  The 3-cochain  $\mu_1$ in $C^3_G(T,T)$ is called infinitesimal of the equivariant  deformation $\mu_t$. In general, if $\mu_i=0,$ for $1\le i\le n-1$, and $\mu_n$ is a nonzero cochain in  $C^3_G(T,T)$, then $\mu_n$ is called n-infinitesimal of the equivariant deformation $\mu_t$.
\end{defn}
\begin{prop}
  The infinitesimal   $\mu_1$ of the equivariant deformation  $\mu_t$ is a 3-cocycle in $C^3_G(T,T).$ In general, n-infinitesimal  $\mu_n$ is a 3-cocycle in $C^3_G(T,T).$
\end{prop}
\begin{proof}
  For n=1, proof is obvious from the Remark \ref{rbrem1}. For $n>1$, proof is similar.
\end{proof}
We can write Equations \ref{rbeqn1}, \ref{rbeqn2} and \ref{rbeqn3} for $r=n+1$ using the definition of coboundary $\delta$ as
\begin{equation}\label{rbeqn4}
   \mu_{n+1}(a,a,b)=0,
  \end{equation}
  \begin{equation}\label{rbeqn5}
    \mu_{n+1}(a,b,c)+\mu_{n+1}(b,c,a)+\mu_{n+1}(c,a,b)=0,
   \end{equation}
   \begin{eqnarray}\label{rbeqn6}
      &&\delta \mu_{n+1}(a,b,c,d,e)\notag \\
      &=& \sum_{\substack{ i+j=n+1\\i,j>0}}\mu_i(a,b,\mu_j(c,d,e))-\sum_{\substack{i+j=n+1\\i,j>0}}\{\mu_i(\mu_j(a,b,c),d,e)+ \mu_i(c,\mu_j(a,b,d),e) \notag\\
       && +\mu_i(c,d,\mu_j(a,b,e))\}\notag\\
   \end{eqnarray}
 for all  $a,b,c,d,e\in T$.

Define a 5-cochain $F_{n+1}$ by
\begin{align*}
&F_{n+1}(a,b,c,d,e)\notag\\
&= \sum_{\substack{i+j=n+1\\i,j>0}}\mu_i(a,b,\mu_j(c,d,e))-\sum_{\substack{i+j=n+1\\i,j>0}}\{\mu_i(\mu_j(a,b,c),d,e)+ \mu_i(c,\mu_j(a,b,d),e)\notag \\
       & +\mu_i(c,d,\mu_j(a,b,e))\}\\
\end{align*}

\begin{lem}\label{Obsl1}
  The 5-cochain $F_{n+1}$ is invariant, that is $F_{n+1}\in C_G^5(T,T).$
\end{lem}
\begin{proof}
 To prove that $F_{n+1}$ is invariant we show that  $$F_{n+1}(ga,gb,gc,gd,ge)=gF_{n+1}(a,b,c,d,e)$$  for all $a,b,c,d,e\in T$ . From Definition \ref{rb2}, we have $$\mu_i(ga,gb,gc)=g\mu_i(a,b,c),$$ for all $a,b,c\in T.$
 So, we have, for all $a,b,c,d,e\in T$,
 {\scriptsize \begin{align*}
   &F_{n+1}(ga,gb,gc,gd,ge)\\
   &= \sum_{\substack{i+j=n+1\\i,j>0}}\mu_i(ga,gb,\mu_j(gc,gd,ge))-\sum_{\substack{i+j=n+1\\i,j>0}}\{\mu_i(\mu_j(ga,gb,gc),gd,ge)+ \mu_i(gc,\mu_j(ga,gb,gd),ge)\\
 & +\mu_i(gc,gd,\mu_j(ga,gb,ge))\} \\
   &=  \sum_{\substack{i+j=n+1\\i,j>0}}\mu_i(ga,gb,g\mu_j(c,d,e))-\sum_{\substack{i+j=n+1\\i,j>0}}\{\mu_i(g\mu_j(a,b,c),gd,ge)+ \mu_i(gc,g\mu_j(a,b,d),ge)\\
& +\mu_i(gc,gd,g\mu_j(a,b,e))\} \\
    &=  g\sum_{\substack{i+j=n+1\\i,j>0}}\mu_i(a,b,\mu_j(c,d,e))-g\sum_{\substack{i+j=n+1\\i,j>0}}\{\mu_i(\mu_j(a,b,c),d,e)+ \mu_i(c,\mu_j(a,b,d),e) \\
 & +g\mu_i(c,d,\mu_j(a,b,e))\} \\
   &=gF_{n+1}(a,b,c,d,e).
 \end{align*}}
 So we conclude that $F_{n+1}\in C_G^5(T,T).$
\end{proof}
\begin{defn}
  The 5-cochain $F_{n+1}\in C_G^5(T,T)$ is called $(n+1)th$ obstruction cochain for extending a given equivariant deformation of order n to an equivariant deformation of T of order $(n+1)$. Now onwards we denote $F_{n+1}$ by $Ob_{n+1}(T)$
\end{defn}
 By using Lemma \ref{Obsl1} and \cite{Kub-Tani}, we have the following result.
\begin{thm}
  The $(n+1)$th obstruction cochain $Ob_{n+1}(T)$ is a 5-cocycle.
\end{thm}
\begin{thm}
Let $\mu_t$ be an   equivariant deformation of T of order n. Then $\mu_t$ extends to an equivariant deformation of order $n+1$ if and only if cohomology class of $(n+1)$th obstruction $Ob_{n+1}(T)$  vanishes.
\end{thm}
\begin{proof}
  Suppose that an equivariant deformation $\mu_t$ of T of order n extends to an equivariant deformation of order $n+1$. This implies that \ref{rbeqn1},\ref{rbeqn2} and \ref{rbeqn3} are satisfied for $r=n+1.$  Observe that this implies $Ob_{n+1}(T)=\delta^3\mu_{n+1}$. So cohomology class of  $Ob_{n+1}(T)$ vanishes. Conversely, suppose that  cohomology class of  $Ob_{n+1}(T)$ vanishes, that is $Ob_{n+1}(T)$ is a coboundary. Let
  $$ Ob_{n+1}(T)=\delta^3\mu_{n+1},$$
  for some 3-cochain  $\mu_{n+1}\in C^3_G(T,T).$ Take
  $$\tilde{\mu_t}=\mu_t+\mu_{n+1}t^{n+1}.$$
  Observe that $\tilde{\mu_t}$ satisfies  \ref{rbeqn1},\ref{rbeqn2} and \ref{rbeqn3} for $0\le r\le n+1$. So  $\tilde{\mu_t}$ is an equivariant extension of $\mu_t$ and is of order $n+1$.

\end{proof}
\begin{cor}
  If $H^5_G(T)=0$, then every 3-cocycle in $C^3_G(T,T)$ is an infinitesimal of some equivariant deformation of $T.$
\end{cor}

\section{Equivalence of equivariant deformations and rigidity }\label{rbsec5}
Let   $\mu_t$  and $\tilde{\mu_t}$ be two equivariant deformations of T. An equivariant formal isomorphism from the equivariant deformations $\mu_t$ to $\tilde{\mu_t}$ of a Lts T is a $k[[t]]$-linear G-automorphism $\Psi_t:T[[t]]\to T[[t]]$ of the  form  $\Psi_t=\sum_{i\ge 0}\psi_it^i$, where each $\psi_i$ is an equivariant $k$-linear map $T\to T$, $\psi_0(a)=a$, for all $a\in T$ and $\tilde{\mu_t}(\Psi_t(a),\Psi_t(b),\Psi_t(c))=\Psi_t\mu_t(a,b,c),$ for all $a,b,c\in T.$
\begin{defn}
  Two equivariant deformations $\mu_t$  and $\tilde{\mu_t}$ are said to be equivalent if there exists an equivariant formal isomorphism  $\Psi_t$ from $\mu_t$ to  $\tilde{\mu_t}$.
\end{defn}
Equivariant formal isomorphism on the collection of all  equivariant deformations of a Lts T is an equivalence relation.
\begin{defn}
  Any equivariant deformation of T that is equivalent to the deformation $\mu_0$ is said to be a trivial deformation.
\end{defn}

\begin{thm}
  The cohomology class of the infinitesimal of an equivariant  deformation $\mu_t$ of  a Lts T is determined by the equivalence class of $\mu_t$.
\end{thm}
\begin{proof}
  Let  $\Psi_t$ from  $\mu_t$ to  $\tilde{\mu_t}$ be an equivariant  formal  isomorphism. So, we have  $\tilde{\mu_t}(\Psi_ta,\Psi_tb,\Psi_tc)=\Psi_t\circ \mu_t(a,b,c),$ for  all $a,b,c\in T$. This implies that $(\mu_1-\tilde{\mu_1})(a,b,c)=[\psi_1abc]+[a\psi_1bc]+[ab\psi_1c]-\psi_1[abc]$. So we have $\mu_1-\tilde{\mu_1}=\delta^1\psi_1$ This completes the proof.
\end{proof}
\begin{defn}
  An equivariant Lts T is said to be rigid if every deformation of $\mu_t$ of T  is trivial.
\end{defn}
\begin{thm}\label{rb-100}
  A non-trivial equivariant deformation of a Lts is equivalent to an equivariant  deformation whose n-infinitesimal is not a coboundary, for some $n\ge 1.$
\end{thm}
\begin{proof}
  Let $\mu_t$ be an equivariant deformation of  a Lts T with n-infinitesimal $\mu_n$, for some $n\ge 1.$ Assume that there exists a 1-cochain $\psi\in C_G^1(T,T)$ with $\delta^1\psi=\mu_n.$  Take $\Psi_t=Id_T+\psi t^n$. Define $\tilde{\mu_t}=\Psi_t\circ \mu_t\Psi_t^{-1}$. Clearly, $\tilde{\mu_t}$ is an equivariant deformation of T and $\Psi_t$ is an   equivariant  formal isomorphism from $\mu_t$  to  $\tilde{\mu_t}$. For $u,v,w\in T,$ we have  $\tilde{\mu_t}(\Psi_tu,\Psi_tv, \Psi_tw)=\Psi_t(\mu_t(u,v,w)),$ which implies $\tilde{\mu_i}=0,$ for $1\le i\le n.$ So $\tilde{\mu_t}$ is equivalent to the given deformation and $\tilde{\mu_i}=0,$ for $1\le i\le n.$ We can repeat the argument to get rid off any infinitesimal that is a coboundary. So the process must stop if the deformation is nontrivial.
\end{proof}
An immediate consequence of the Theorem \ref{rb-100} is following corollary.
\begin{cor}
  If $H_G^3(T,T)=0,$ then T is rigid.
\end{cor}

%\section*{Acknowledgements}

%\section*{References}
%
%\bibliography{rb2}

\end{document}